\documentclass[11pt]{amsart}

\usepackage{amsfonts,amssymb,verbatim,amsmath,amsthm,latexsym,textcomp,amscd}
\usepackage{graphicx}
\usepackage{color}
\usepackage{url}
\usepackage{array}
\usepackage{enumitem}
\usepackage{hyperref}
\usepackage{todonotes}

\newcommand{\lipnorm}[1]{\|#1\|_{\mathrm{Lip}^w_+}}
\DeclareMathOperator{\Var}{Var}

\newcommand{\LipWp}{\mathrm{Lip}^w_+}
\DeclareMathOperator{\Unif}{Unif}

\newtheorem*{main-thm}{Main Theorem}
\newtheorem{theorem}{Theorem}[section]
\newtheorem{lemma}[theorem]{Lemma}
\newtheorem{proposition}[theorem]{Proposition}
\newtheorem{corollary}[theorem]{Corollary}
\theoremstyle{definition}
\newtheorem{definition}[theorem]{Definition}
\theoremstyle{remark}
\newtheorem{remark}[theorem]{Remark}

\title[Spatial limit theorem for interval exchange transformations]
{Spatial limit theorem\\ for interval exchange transformations}

\author{Alexey Klimenko}
\address{\begin{flushleft} Alexey Klimenko\\
\upshape Steklov Mathematical Institute of Russian Academy of Sciences,\\Gubkina str. 8,  119991, Moscow, Russia;\\
and National Research University Higher School of Economics,\\ Usacheva str. 6, 119048,  Moscow, Russia \\
\texttt{klimenko@mi.ras.ru} \end{flushleft}}

\date{}

\begin{document}

 \begin{abstract}
We prove a spatial limit theorem for generic interval exchange transformations (IETs): for a generic IET the normalized ergodic sums of a sufficiently regular (e.g., Lipschitz) function have the same asymptotic behavior of distributions as the behaviour of ergodic integrals for a generic translation flow on a flat surface, described by A. Bufetov.
\end{abstract}

\maketitle

\section{Introduction}

Let $T\colon X\to X$ be a transformation of a probability space $(X,\mu)$ that conserves the measure $\mu$. For a function $\varphi\in L^1(X,\mu)$ consider its \emph{ergodic sums}
\begin{equation*}
S_n\varphi(x)=\sum_{k=0}^{n-1} \varphi(T^kx).
\end{equation*}
One can regard $S_n\varphi$ as a random variable on the probability space $(X,\mu)$, and we are interested in asymptotical behavior of the \emph{distributions} of
appropriately normalized random variables $S_n\varphi$ as $n\to\infty$. Following Dolgopyat and Sarig \cite{DolSar}, we refer to results of this type as \emph{spatial (distributional) limit theorems}.

Our main result states spatial limit theorem for a generic interval exchange transformation.
To state precisely what genericity means here, we need some notation for translation flows on flat surfaces, which are continuous-time counterparts of IETs. Here we follows Bufetov's paper \cite{Bufetov-Ann}.

In a moduli space of Abelian differentials let $\mathcal{M_\kappa}$ be a stratum of surfaces such that the orders of their singularities form a tuple $\kappa=(\kappa_1,\dots,\kappa_\sigma)$, endowed with a choice of a ``horizontal'' separatrix of one of the singularities. 
Let $\mathcal H$ be a connected component of $\mathcal M_k$ and let
$\mathbf{g}^s$ be the Teichm\"uller flow on $\mathcal H$. Finally, consider a $\mathbf{g}^s$-invariant ergodic probability measure $\mathbb P$ on $\mathcal H$.

For every surface $M\in \mathcal M_\kappa$ there is a unique decomposition of $M$ into Markov rectangles such that the bases of these rectangles belong to a segment~$I$ of the marked separatrix starting at the singular point at the end of this separatrix, the length of this segment is at least one, while the application of the Rauzy induction makes it less than one.
Therefore, the first return map to the segment $I$ for the vertical flow $h^+_t$ on $M$ yields an interval exchange transformation $T_{M}$. The number of segments for $T_M$ is equal to the number of rectangles, and hence is fixed. Moreover, permutations of segments for $T_{M}$, $M\in\mathcal H$ belong to the same Rauzy class almost surely. Indeed, $T_{M}$ and $T_{\mathbf{g}^sM}$ differ by a scaling and several applications of the Rauzy induction, hence their permutations belong to the same Rauzy class, so this Rauzy class is $\mathbf{g}$-invariant. But $\mathbb P$ is ergodic, so this Rauzy class should be the same almost surely.

Therefore, one can consider a set $\widehat{\mathcal{H}}=\{T_M, M\in\mathcal H\}$, and the measure $\mathbb P$ induces a measure $\widehat{\mathbb{P}}$ on $\widehat{\mathcal{H}}$.

A. Bufetov \cite{Bufetov-Ann} proved a limit theorem for $\mathbb{P}$-almost all translation froms provided some conditions on $\mathbb{P}$ hold.
Namely, for sufficiently regular function $\psi$ on $M$ outside of a finite-codimension ``degenerate'' subspace, the distributions of random variables
\begin{equation*}
\widehat{I}_t\psi=\frac{I_t\psi -\mathbb EI_t\psi}{\sqrt{\operatorname{Var}I_t\psi}},\quad\text{where}\quad I_t\psi(x)=\int_0^t \psi(h^+_s(x))\,ds,
\end{equation*}
have no limit as $t\to\infty$ but its set of limit points in the space of distributions can be described, see Section~\ref{sec:bufetov} for the detailed statement.

\begin{main-thm}
Assume that a $\mathbf{g}^s$-ergodic measure $\mathbb{P}$ on a connected component $\mathcal{H}$ of a stratum of Abelian differentials satisfy conditions of Bufetov's theorem. Then for $\widehat{\mathbb P}$-almost every $T\in\widehat{\mathcal{H}}$ and for a regular function $f$ on the segment $I$ where $T$ acts if $f$ lies outside of a finite-codimensional subspace, then the sequence of distributions of the random variables 
\begin{equation}
 \widehat{S}_n f=\frac{S_n f -\mathbb ES_n f}{\sqrt{\Var S_n f}}
\end{equation}
have the same limit behaviour as the behaviour of distributions of ergodic integral for some $M\in\mathcal{H}$ with $T_M=T$ and some function $\psi$ on $M$.
\end{main-thm}

The structure of the paper is the following. We discuss Bufetov's limit theorem for translation flows in the next section, then in Section~\ref{sec:sums-intls} we give the detailed statement of our Main Theorem, which is then proved in this section and the last Section~\ref{sec:unif-distr}.

\section{Bufetov's limit theorem for translation flows}
\label{sec:bufetov}

In this section we mostly follows \cite[Section 1]{Bufetov-Ann}.
 
A bounded measurable function $\varphi$ on a flat surface $M$ is called \emph{weakly Lipschitz} if there exists a constant $C>0$ such that for any $x\in M$ and any $t_1,t_2>0$ such that the rectangle
$\Pi(x,t_1,t_2)=\{h^+_t(h^-_s(x)): 0\le t\le t_1, 0\le s\le t_2\}$ is \emph{admissible}, i.e. does not contain singular points in its interior, we have
\begin{equation}\label{eq:w-lip}
\left|\int_0^{t_1}\varphi(h^+_t(x))\,dt-\int_0^{t_1}\varphi(h^+_t(h_{t_2}^-(x)))\,dt\right|\le C.
\end{equation}
The set $\LipWp(M)$ of all weakly Lipschitz functions is a normed space with the norm defined by
\begin{equation*}
\lipnorm{\varphi}=\sup_M |\varphi|+C_\varphi,
\end{equation*}
where $C_\varphi$ is the infimum of all $C$ satisfying \eqref{eq:w-lip}.

The main tool to describe behaviour of ergodic integrals for a translation flow is the following \emph{Bufetov cocycles}.

\begin{definition}
	A function $\Phi^+$ on arcs of the vertical flow $h_t^+$ on $M$ is called \emph{vertical Bufetov cocycle} if the following holds.\\
	1) $\Phi^+$ is additive: if $\Phi^+(x,t)$ denotes its value on the arc $h^+_{[0,t]}(x)=\{h^+_s(x):0\le s\le t\}$, then
	\begin{equation*}
		\Phi^+(x,t_1+t_2)=\Phi^+(x,t_1)+\Phi^+(h^+_{t_1}(x),t_2).
	\end{equation*}
	2) $\Phi^+$ is H\"older: there exists $t_0>0$, $\theta>0$ such that $|\Phi^+(x,t)|<t^\theta$ for all $x\in M$, $|t|\le t_0$.\\
	3) $\Phi^+$ is holonomy invariant: if $\Pi(x,t_1,t_2)$ is admissible, then $\Phi^+(x,t_1)=\Phi^+(h^-_{t_2}(x),t_1)$.\\
	The space of all such cocycles is denoted by $\mathfrak B^+=\mathfrak B^+(M)$. One can also define the symmetric class of horizontal Bufetov cocycles $\mathfrak B^-$.  
\end{definition}
The spaces $\mathfrak B^\pm(M)$ are isomorphic to the spaces of Kontsevich---Zorich cocycles\cite{KonZor, Zorich} of the surface~$M$. 
Similar functional objects also arose earlier in the work of G. Forni \cite{Forni}. 

The base for the study of the Bufetov cocycles is matrix cocycles given by the incidence matrices of the Markov partition of the surface into rectangles and the same matrices for partitions produced from it by forward or backward Rauzy induction. These matrix cocycles $\mathbb{A}$ on the space $(\mathcal H,\mathbb P)$ are assumed to satisfy conditions of Oseledets---Pesin reduction theorem (for details see \cite[Assumption 2.3]{Bufetov-Ann}).

Bufetov showed that the spaces $\mathfrak B^\pm$ have finite dimensions and there exists a nondegenerate pairing $\langle{}\cdot{},{}\cdot{}\rangle:\mathfrak B^+\times \mathfrak B^-\to\mathbb C$. For a cocycle $\Phi^-$ one can define a functional $m_{\Phi^-}$ on the space of weakly Lipschitz functions as follows:
we decompose $M$ into admissible rectangles, so $m_{\Phi^-}(\varphi)$ is the sum over all these rectangles $\Pi=\Pi(x,t_1,t_2)$ of the integrals $\int_0^{t_2}\int_0^{t_1}
f(h^+_t(h_{s}^-(x)))\,dt\,d\Phi^-(s)$, where the external integral can be defined as a limit of Riemann sums. 
The map $\varphi\mapsto m_{\Phi^-}(\varphi)$ is linear in $\Phi^-$ so the pairing between $\mathfrak B^+$ and $\mathfrak B^-$ allows us to define the functional $\Phi^+_\varphi$ by the formula
\begin{equation*}
	m_{\Phi^-}(\varphi)=\langle \Phi^+_\varphi,\Phi^-\rangle\quad\text{for any}\quad \Phi^-\in\mathfrak B^-.
\end{equation*}

The Teichm\"uller flow on $\mathcal H$ naturally acts on the fiber bundles of Bufetov's cocycles over it.
Hence one can decompose the space $\mathfrak B^+(M)$ into its Oselelets subspaces:
\begin{equation}\label{eq:decomp}
	\mathfrak B^+(M)=\bigoplus_{i=1}^l \mathfrak B^+_i(M)
\end{equation}
corresponding to the Lyapunov exponents $1=\theta_1>\theta_2\dots>\theta_l>0$.
The space $\mathfrak B^+_1(M)$ is the linear hull of the cocycle $\Phi^+_1(x,t)=t$.

Bufetov proved the following asymptotic decomposition of the ergodic integral for the generic translation flows.

\begin{theorem}[{\cite[Theorem 1]{Bufetov-Ann}}]
	Let $\mathbb P$ be an ergodic $\mathbf{g}$-invariant probability measure on $\mathcal H$. 
	Then for any $\varepsilon>0$ there exists $C_\varepsilon>0$ such that for $\mathbb P$-almost every $M\in \mathcal H$, any $\varphi\in\LipWp(M)$, any $x\in M$, and any $T>0$ we have
	\begin{equation*}
		\left|\int_0^T \varphi(h_t^+(x))\,dt-\Phi^+_\varphi(x,T)\right|\le C_\varepsilon\lipnorm{\varphi}(1+T^\varepsilon).
	\end{equation*}  
\end{theorem}

Later we will need the following truncated version of this formula: let $\Phi^{+}_{\varphi,\le r}$ be a projection of $\Phi^+_\varphi$ onto the first $r$ terms $\mathfrak B^+_1\oplus\dots\oplus \mathfrak B^+_r$ in the Oseledets decomposition. Then 
\begin{equation}\label{eq:buf-asymp-r}
\left|\int_0^T \varphi(h_t^+(x))\,dt-\Phi^{+}_{\varphi,\le r}(x,T)\right|\le C_\varepsilon\lipnorm{\varphi}(1+T^{\theta_r-\varepsilon}).
\end{equation}  
for some $\varepsilon>0$.

We now pass to Bufetov's limit theorem. Informally speaking, it states that the distribution of the normalized ergodic integral for $\varphi$ tends to the distribution of the first nonzero component of $\Phi^+_\varphi$. Let $\Phi^+_\varphi=\sum_{i=1}^l \Phi^+_{\varphi,i}$, where $\Phi^+_{\varphi,i}\in\mathfrak B^+_i$, then denote $i(\varphi)=\min\{i=2,\dots,l:\Phi^+_{\varphi,i}\ne 0\}$.

The Teichm\"uller flow on the moduli space induces natural action on the bundle $\{\mathfrak{B}^+(M), M\in\mathcal H\}$. Namely, for $\Phi^+\in\mathfrak{B}^+(M)$ and $x\in M$ let $y\in\mathbf{g}^s M$ be the image of $x$ and denote 
\begin{equation}
\label{eq:teich-scale}
(\mathbf{g}^s\Phi^+)(y,t):=\Phi^+(x,e^{s}t),
\end{equation}
then we have $\mathbf{g}^s\Phi^+\in \mathfrak{B}^+(\mathbf{g}^sM)$. Moreover, 
the decomposition \eqref{eq:decomp} is invariant under this action, and we denote the action of $\mathbf{g}^s$ on the $\mathfrak{B}^+_i$ by $\mathbf{g}^s_i$.

As one can see from the approximation theorem, the distribution of $\widehat{I}_t\psi$ is close to the distribution of $\Phi^+_\psi({}\cdot{},t)$, and more precisely, to the distribution $\Phi^+_{\psi,i(\psi)}({}\cdot{},t)$. The scaling~\eqref{eq:teich-scale} allow us to identify the last distribution with the one of $(\mathbf{g}^{\log t}\Phi^+_{\psi,i(\psi)})({}\cdot{},1)$. This yields the following theorem.

\begin{theorem}[{\cite[Theorem 2]{Bufetov-Ann}}]\label{thm:buf-limit}
Let the above-mentioned matrix cocycles satisfy the conditions of Oseledets---Pesin theorem with respect ot an ergodic $\mathbf{g}^s$-invariant probability measure~$\mathbb P$. Then there exists a constant $\alpha>0$ and a measurable function $C\colon \mathcal H\times\mathcal H\to\mathbb R_+$ such that for $\mathbb P$-almost every $M\in\mathcal H$ and every $\psi\in \LipWp(M)$ such that $\Phi^+_\psi\notin\mathfrak{B}^+_1(M)$ (and hence $i(\psi)$ is well defined), we have
\begin{equation}\label{eq:buf-lim-th}
d\bigl(\operatorname{Law}(\widehat{I}_T(\psi)),
\operatorname{Law}((\mathbf{g}_{i(\psi)}^{\log T}\widehat{\Phi}^+_{\psi,i(\psi)})({}\cdot{},1))\bigr)\le C(M,\mathbf{g}^{\log T} M) T^{-\alpha},
\end{equation}
where $d$ is either Kantorovich---Rubinstein or L\'evy---Prohorov distance on the space of distributions.
\end{theorem}

Note that in the case of simple Lyapunov exponents the formula \eqref{eq:buf-lim-th} has a simpler form: $\operatorname{Law}((\mathbf{g}_{i(\psi)}^{\log T}\widehat{\Phi}^+_{\psi,i(\psi)})({}\cdot{},1))$ is the normalized distribution of \emph{any} nonzero cocycle in $\mathfrak{B}^+_{i(\psi)}(\mathbf{g}^{\log T}M)$ since the normalization depends only on the projectivization of $\mathbf{g}^{\log T}({\Phi}^+_{\psi,i(\psi)})$, and $\mathfrak{B}^+_{i(\psi)}$ is one-dimensional.

\begin{remark}
	In particular, Bufetov's theorems applies to the Masur---Veech ``smooth'' measure~$\mathbb P$, \cite{Masur,Veech}. In this case the corresponding measure~$\widehat{\mathbb{P}}$ on the space of interval exchange transformations is equivalent to the Lebesgue measure. As Avila and Viana showed in~\cite{AvilaViana}, the second Lyapunov exponent for the Masur---Veech measure is simple.
	
	Let us also recall the recent result by Chaika and Eskin~\cite{ChaiEsk}. It implies that for \emph{any} given flat surface the set of directions such that the corresponding translation flow satisfies assumptions of Bufetov's theorems has full Lebesgue measure on the circle. 
\end{remark}

\section{Main theorem and approximation of ergodic sums via integrals}
\label{sec:sums-intls}

Let us give the precise statement for the main theorem of this paper.

For $\widehat{\mathbb{P}}$-almost every interval exchange transformation~$T$ there exists a surface $M\in\mathcal{H}$ with $T=T_M$ such that Bufetov's theorems hold for $M$.
The surface $M$ is decomposed into several rectangles of heights $h_i$ with bases on the continuity segments $I_i$ of $T$. 
Introduce a coordinate system $(x,y)$ on the union of these rectangles so that $x$ is a ``horizontal'' coordinate and $y$ is a ``vertical'' one, i.e. the flow $h^+_t$ corresponds to the differential equations $\dot x=0$, $\dot y=1$.
In other words, the flow $h_+^t$ on $M$ is the special flow over $T$ with the roof function
\begin{equation*}
h(x)=h_i\text{ for }x\in I_i.
\end{equation*}
For any function $f$ on $I$ define the following function $\psi_f$ on $M$.
\begin{equation}\label{eq:psi-f}
\psi_f(x,y)=\frac{f(x)}{h(x)}.
\end{equation}

\begin{theorem}[Main theorem restated]\label{thm:var-main}
	Let $\mathbb P$ be an ergodic $\mathbf{g}^s$-invariant probability measure on a connected component $\mathcal H$ of the space of Abelian differentials and assume that $\mathbb P$ satisfies the conditions of Bufetov's limit theorem.
	Then for any $\varepsilon>0$ there exists a measurable function $C\colon \mathcal H\times\mathcal H\to\mathbb R_+$ such that the following holds.
	
	Take any $T\in\widehat{\mathcal{H}}$ such that there exists $M\in\mathcal H$ with $T=T_M$ such that the conditions of Bufetov's approximation theorem hold for $M$ (recall that this takes place for $\widehat{\mathbb P}$-almost all $T$). Then for all functions $f$ on $I$ such that $\psi_f\in\LipWp(M)$ and $\Phi^+_{\psi_f,2}\ne 0$ we have
	\begin{align*}
	d_{KR}\bigl(\operatorname{Law}(\widehat{S}_nf),
	\operatorname{Law}((\mathbf{g}_{2}^{\log n}\widehat{\Phi}^+_{\psi_f,2})({}\cdot{},1))\bigr)&\le C(M,\mathbf{g}^{\log n} M) n^{\theta_2^2-\theta_2+\varepsilon},\\
	d_{LP}\bigl(\operatorname{Law}(\widehat{S}_nf),
	\operatorname{Law}((\mathbf{g}_{2}^{\log n}\widehat{\Phi}^+_{\psi_f,2})({}\cdot{},1))\bigr)&\le C(M,\mathbf{g}^{\log n} M) n^{\frac23\bigl(\theta_2^2-\theta_2\bigr)+\varepsilon}.
	\end{align*}
\end{theorem}

\begin{remark}
	Note that if $f$ is Lipschitz on each continuity segment~$I_i$ of the interval exchange transformation~$T$, then $\psi_f\in\LipWp(M)$. Indeed, the function $\psi_f$ is Lipschitz continuous inside any admissible rectangle with the same constant $L_f\cdot(1/\min(h_j))$, where $L_f$ is the Lipschitz constant for $f$. It remains to observe that the product of the height and the width of any admissible rectangle does not exceed the area of $M$, which equals~$1$. 
\end{remark}

The first step of the proof is to approximate an ergodic sum for the interval exchange transformation $T\colon I\to I$ by an ergodic integral for translation flow on some surface $M$ such that $T=T_M$. Until the end of the paper we assume that the setting and the conditions of Theorem~\ref{thm:var-main} hold.

\begin{proposition}
	1. Let $t_n(x)$ be the time when the point $(x,0)$ under the action of the flow $h^+_t$ makes its $n$-th return to the transversal $I=\{y=0\}$.
	Then one have
	\begin{equation}\label{eq:Sn-and-It}
	S_n f(x)=I_{t_n(x)}\psi_f(x,0).
	\end{equation}
	2. $\displaystyle \int_If(x)\,dx=\int_M \psi_f(x,y)\,dxdy$.
\end{proposition}

\begin{proof}
	1. Observe that
	\begin{equation*}
		\int_0^{t_1(x)}\psi_f(h_t^+(x,0))\,dt=\int_0^{h(x)}\frac{f(x)}{h(x)}\,dt=f(x).
	\end{equation*}
	Now
	\begin{equation*}
	I_{t_n(x)}\psi_f(x,0)=\sum_{k=0}^{n-1}\int_{t_k(x)}^{t_{k+1}(x)} \psi_f(h_t^+(x,0))\,dt=\sum_{k=0}^{n-1}\int_0^{h(T^k(x))}\frac{f(T^k x)}{h(T^k x)}dt=S_nf(x).
	\end{equation*}
	2. The integrand in the right-hand side of the equality does not depend on~$y$, and its integration in~$y$ yields the left-hand side.  
\end{proof}

\begin{proposition}\label{prop:tn-asymp}
  Let $\psi_1$ be the function on $M$ constructed by~\eqref{eq:psi-f} from the constant function~$1$ on~$I$. Then for any $\delta>0$ there exists $A_\delta$ such that 
  \begin{equation*}
  	\Bigl| t_n(x)-\bigl[n-\Phi^+_{\psi_1,2}((x,0),n)\bigr]\Bigr|\le A_\delta(1+n^{\theta_2+\delta}).
  \end{equation*}	
\end{proposition}

\begin{remark}
	Later we will use the $O$-notation for such estimates. The constants in $O(\,{\cdot}\,)$ can depend on the surface $M$ and small parameters $\delta$, $\varepsilon$, etc.
\end{remark}

\begin{proof}
Applying the formula~\eqref{eq:Sn-and-It} to the function $f\equiv 1$, we get the equation
\begin{equation}\label{eq:n-and-psi1}
n=I_{t_n(x)}\psi_1(x,0),
\end{equation}
and we will find asymptotics for its solution.

First of all, \eqref{eq:buf-asymp-r} with $r=1$ gives that for any $\alpha>0$ there exists $C>0$ such that for all $t>0$ we have
\begin{equation*}
|I_t\psi_1-\overline{\psi_1}t|\le \alpha t+C,\quad\text{where}\quad \overline{\psi_1}=\int_M\psi_1\,dxdy=1.
\end{equation*}
Therefore,
\begin{equation*}
	(1-\alpha)t_n(x)-C\le n\le (1+\alpha)t_n(x) +C,
\end{equation*}
hence, conversely, for any $\beta>0$ and some $D>0$ we have 
\begin{equation*}
(1-\beta)n-D\le t_n(x)\le (1+\beta)n +D.
\end{equation*}
Further, the formula~\eqref{eq:buf-asymp-r} with $r=1$ gives us that
\begin{equation*}
|n-t_n(x)|=|n-\Phi^+_{\psi_1,1}((x,0),t_n(x))|
\le C_\varepsilon (1+t_n(x)^{\theta_2+\varepsilon})\le C'_\varepsilon (1+n^{\theta_2+\varepsilon}).
\end{equation*} 
Combining this estimate with \eqref{eq:buf-asymp-r} for $r=2$ we have
\begin{multline*}
	|n-t_n(x)-\Phi^+_{\psi_1,2}((x,0),n)|\le
	|n-\Phi^+_{\psi_1,\le 2}((x,0),t_n(x))|+
	|\Phi^+_{\psi_1,2}(h^+_n((x,0)),t_n(x)-n)|\\
	{}\le 
	C''_\varepsilon (1+n^{\theta_3+\varepsilon})+\tilde C_\varepsilon (1+|t_n(x)-n|^{\theta_2+\varepsilon})\le \tilde C'_\varepsilon (1+n^{(\theta_2+\varepsilon)^2}).
\end{multline*}
Hence we have obtained that for any $\delta>0$ there exists $A_\delta>0$ such that
\begin{equation*}
	\bigl| t_n(x)-(n-\Phi^+_{\psi_1,2}((x,0),n))\bigr|\le A_\delta (1+n^{(\theta_2^2+\delta)}).
\end{equation*}
Now choosing $\varepsilon$ such that $(\theta_2+\varepsilon)^2\le\theta_2^2+\delta$ we obtain the desired estimate.
\end{proof}

\begin{lemma}\label{lem:approx}Denote $\overline f=(1/|I|)\int_I f(x)\,dx$.
	Then the  following asymptotic estimate holds for $S_n(f)$: for any $\delta>0$
	\begin{equation*}
	S_n(f)=\overline f\cdot n + \Phi^+_{(\psi_f- \overline f \psi_1),2}((x,0),n)+ O(n^{\theta_2^2+\delta})\lipnorm{\psi_f}.
	\end{equation*}
\end{lemma}

\begin{proof}
	Applying Proposition~\ref{prop:tn-asymp} and formula~\eqref{eq:buf-asymp-r} with $r=2$ to the identity~\eqref{eq:Sn-and-It} we have
	\begin{equation}\label{eq:Snf-decomp-lem}
		S_nf(x)=\overline f\cdot n - \overline f \cdot \Phi^+_{\psi_1,2}((x,0),n)+
		\Phi^+_{\psi_f,2}((x,0),t_n(x))+O(n^{\theta_2^2+\delta}).
	\end{equation}
	The third term in the right-hand side of this formula is approximated as follows:
	\begin{multline*}
		\Phi^+_{\psi_f,2}((x,0),t_n(x))-\Phi^+_{\psi_f,2}((x,0),n)=\Phi^+_{\psi_f,2}(h^+_n(x,0),t_n(x)-n)={}\\{}=
		O\bigl((t_n(x)-n)^{\theta_2+\varepsilon}\bigr)\lipnorm{\psi_f}=
		O(n^{(\theta_2+\varepsilon)^2})\lipnorm{\psi_f}.
	\end{multline*}
	Substituting this estimate with $\varepsilon$ such that $(\theta_2+\varepsilon)^2<\theta_2^2+\delta$ into \eqref{eq:Snf-decomp-lem}, we obtain the statement of the lemma. 
\end{proof}

Later we will use the following extension of this lemma.

\begin{corollary}\label{cor:approx}
	Under the assumptions of Lemma~\ref{lem:approx} for any $\varepsilon>0$ one has that for any $\gamma\in [0,1]$
	\begin{equation}\label{eq:Snf-decomp-cor}
	S_n(f)=\overline f\cdot n + \Phi^+_{(\psi_f- \overline f \psi_1),2}(h^+_{-\gamma n^{\theta_2-\varepsilon}}(x,0),n)+ O(n^{\theta_2^2+\delta})\lipnorm{\psi_f}.
	\end{equation}
	where the constant in $O(\,{\cdot}\,)$ does not depend on $\gamma$.
\end{corollary}

\begin{proof}
	The difference between asymptotics in~\eqref{eq:Snf-decomp-lem} and \eqref{eq:Snf-decomp-cor} equals
	\begin{equation*}
		\Phi^+_{(\psi_f - \overline f \psi_1),2}(h^+_n(x,0),-\gamma n^{\theta_2-\varepsilon})-
		\Phi^+_{(\psi_f - \overline f \psi_1),2}((x,0),-\gamma n^{\theta_2-\varepsilon})=O((\gamma n^{\theta_2-\varepsilon})^{\theta_2+\varepsilon})\lipnorm{\psi_f},		
	\end{equation*}
	since $\lipnorm{\psi_f}\le\lipnorm{\psi_f-\overline f\psi_1}=\lipnorm{\check f}$, where $\check f(x)=f(x)-\overline f$.
\end{proof}

\section{End of proof: approximation of uniform distribution of initial points}
\label{sec:unif-distr}

Lemma~\ref{lem:approx} yields that the distribution of $S_n[f](x)$ with $x$ uniformly distributed on $I$ (below we denote this as $x\sim\Unif(I)$) is close to that of $\Phi^+_{\psi_f,\le 2}(x,n)$ with $x\sim\Unif(I)$. However, Theorem~\ref{thm:buf-limit} deals with the distribution $\Phi^+_{\psi_f,\le 2}(p,n)$, where $p$ is uniformly distributed on the \emph{whole surface}~$M$. To relate these two distribution we consider the intermediate one, that of $\Phi^+_{\psi_f,\le 2}(q,n)$ with $q=h^+_{-\gamma n^{\theta_2-\delta}}(x)$, where $x\sim\Unif(I)$ and $\gamma\sim\Unif[0,1]$ are independent. Then such $q$ will be distributed almost uniformly on $M$.

We proceed to the formal considerations. Recall the following property of Kantoro\-vich---Rubinstein and L\'evy---Prokhorov distances. (Here and below we write $d(\xi,\eta)$ instead of $d(\operatorname{Law}(\xi),\operatorname{Law}(\eta))$ for brevity.)

\begin{proposition}
	1. $d_{KR,LP}(\xi,\xi+\varepsilon)\le |\varepsilon|$.\\
	2a. $d_{KR}(\xi,(1+\varepsilon)\xi)\le |\varepsilon|\mathbb{E}|\xi|$.\\
	2b. $d_{LP}(\xi,(1+\varepsilon)\xi)\le |\varepsilon|^{2/3}(\operatorname{Var}\xi)^{1/3}\biggl(\dfrac{1+|\varepsilon|}{1-|\varepsilon|}\biggr)^{2/3}$. 
\end{proposition}

\begin{proof}
The first item is clear. To prove item 2a we use ``mass transportation'' interpretation and shift each piece of mass $dm$ from a point $x$ to the point $(1+\varepsilon)x$ yielding work equal to $|\varepsilon x|\,dm$ and its integral over the whole mass is $|\varepsilon|\mathbb{E}|\xi|$.
The last item is proved as follows. Take any Borel set $A\subset\mathbb R$ and split it into $A_{\le\mu}=\{x\in A, |x|\le\mu\}$ and $A_{>\mu}=A\setminus A_{\le\mu}$, with the parameter $\mu$ being chosen below.
Then if $\xi\in A_{\le\mu}$ then $(1+\varepsilon)\xi\in A^\delta_{\le\mu}$ assuming $\delta\le |\varepsilon|\mu$. Similarly,
 $(1+\varepsilon)\xi\in A_{\le\mu}$ yields $\xi\in A^\delta_{\le\mu}$ if $\delta\le |\varepsilon|\mu/(1-|\varepsilon|)$.
Further,
\begin{equation*}
	\mathbb{P}(\xi\in A_{>\mu})\le \frac{\Var\xi}{\mu^2},\qquad
	\mathbb{P}((1+\varepsilon)\xi\in A_{>\mu})\le \frac{(1+|\varepsilon|)^2\Var\xi}{\mu^2}.
\end{equation*}  
Therefore, $d_{LP}(\xi,(1+\varepsilon)\xi)\le\delta$ if
\begin{equation*}
	\frac{(1+|\varepsilon|)^2\Var\xi}{\mu^2}\le\delta,\qquad  \delta\le\frac{|\varepsilon|\mu}{(1-|\varepsilon|)},
\end{equation*}
and these conditions are satisfied with $\delta$ equal to the right-hand side of the inequality in item 2b and $\mu=\delta(1-|\varepsilon|)/|\varepsilon|$.
\end{proof}

\begin{proposition}\label{prop:appr-dens}
	Let $q$ be a random point on the surface $M$ with
	\begin{equation*}
		q=h^+_{-\gamma T}(x),
	\end{equation*}
	where $x\sim\Unif(I)$ and $\gamma\sim\Unif[0,1]$ are independent.
	Then the distribution of $q$ has density $\rho_T$ with respect to the uniform measure on $M$ and we have
	\begin{equation*}
		|\rho_T-1|=O(T^{\theta_2+\varepsilon-1})\text{ for any }\varepsilon>0.
	\end{equation*} 
\end{proposition}

\begin{proof}
	Considering a small rectangle near a point $(x,y)\in M$ one can see that the density $\rho_T(x,y)$ equals $N_T(x,y)/T$, where $N_T(x,y)$ is the number of intersections of the arc segment $h^+_{[0,T]}((x,y))$ with the segment~$I$. Hence we have
	\begin{equation*}
		\biggl|N_T(x,y)-\int_0^T \psi_1(h^+_t((x,y)))\,dt\biggr|\le 1,
	\end{equation*}
	as if we split the integral by the times when the arc crosses $I$, every integral over the internal interval of the partition equals~$1$, and the value of the integral over the first or the last interval belongs to $[0,1]$.
	Now~\eqref{eq:buf-asymp-r} with $r=1$ and $\varphi=\psi_1$ concludes the proof.
\end{proof}

Withoul loss of generality we may assume that $\overline f=0$. Consider the following random variables:
\begin{equation*}
{\arraycolsep=0pt\relax
\begin{array}{ll}
\eta_0=\Phi^+_{\psi_f,2}(p,n),& p\sim\Unif(M),\\[2pt]
\begin{array}{l}
\eta_1=\Phi^+_{\psi_f,2}(h^+_{-\gamma n^{\theta_2-\delta}}(x),n),\hspace*{1em}\\
\eta_2=S_nf(x),
\end{array}
\biggr\}\hspace*{0.5em}& 
\begin{array}{l}
x\sim\Unif(I)\text{ and }\gamma\sim\Unif[0,1]\\
\text{are independent.}
\end{array}
\end{array}}
\end{equation*}
Note that $\eta_2$ does not depend on $\gamma$ but it is convenient to regard $\eta_1$ and $\eta_2$ as random variables on the same probability space. 
Let us also denote
\begin{equation*}
	\widehat{\eta}_k=\frac{\eta_k-\mathbb{E}\eta_k}{\sqrt{\Var\eta_k}}.
\end{equation*}

\begin{lemma}\label{lem:dist-01}For any $\beta>0$ we have
	\begin{align*}
	d_{KR}(\widehat{\eta}_0,\widehat{\eta}_1)&=O(n^{\theta_2^2-\theta_2+\beta}),\\
	d_{LP}(\widehat{\eta}_0,\widehat{\eta}_1)&=O(n^{\frac23\bigl(\theta_2^2-\theta_2\bigr)+\beta}),\\
	\end{align*}
\end{lemma}

\begin{proof}
	In the proof we denote $T=n^{\theta_2-\delta}$. Then the lemma statement takes the form
	\begin{align*}
	d_{KR}(\widehat{\eta}_0,\widehat{\eta}_1)&=O(T^{\theta_2-1+\varepsilon}),\\
	d_{LP}(\widehat{\eta}_0,\widehat{\eta}_1)&=O(n^{\frac23\bigl(\theta_2-1\bigr)+\varepsilon})
	\end{align*}
	for any $\varepsilon>0$.
	
	1. Consider the following random variable
	\begin{equation*}
	\widetilde{\eta}_1=\frac{\eta_1-\mathbb{E}\eta_0}{\sqrt{\Var\eta_0}}
	\end{equation*}
	Then $\widehat{\eta}_0$ and $\widetilde{\eta}_1$ can be regarded as the same function on $M$ but with different measures on $M$, namely, the measures given in Proposition~\ref{prop:appr-dens}. Therefore the distributions of
	$\widehat{\eta}_0$ and $\widetilde{\eta}_1$ are equivalent: $dF_{\widetilde{\eta}_1}=\rho\,dF_{\widehat{\eta}_0}$, and $\rho=1+O(T^{\theta_2-\varepsilon})$.
	Hence
	\begin{equation*}
		|\mathbb E\widetilde{\eta}_1|=|\mathbb E\widetilde{\eta}_1-\mathbb E\widehat{\eta}_0|\le \int |x|\, |\rho(x)-1|\,dF_{\widehat{\eta}_0}\le
		\mathbb E|\widehat{\eta}_0|\cdot O(T^{\theta_2-1+\varepsilon})=O(T^{\theta_2-1+\varepsilon}),
	\end{equation*}
	since $\mathbb E|\widehat{\eta}_0|\le \sqrt{\Var\widehat{\eta}_0}=1$.
	Similarly,
\begin{equation*}
|\mathbb E\widetilde{\eta}_1^2-\mathbb E\widehat{\eta}_0^2|\le \int x^2\, |\rho(x)-1|\,dF_{\widehat{\eta}_0}\le
\mathbb E\widehat{\eta}_0^2\cdot O(T^{\theta_2-1+\varepsilon})=O(T^{\theta_2-1+\varepsilon}),
\end{equation*}
	and thus $\Var\widetilde{\eta}_1=1+O(T^{\theta_2-1+\varepsilon})$.
	
	Further, let us estimate the distances between $\widehat{\eta}_0$ and $\widetilde{\eta}_1$. For L\`evy---Prokhorov distance observe that for any Borel set $A$
	\begin{equation*}
	\frac{\mathbb P(\widehat{\eta}_0\in A)}{\mathbb P(\widetilde{\eta}_1\in A)}=1+O(T^{\theta_2-1+\varepsilon}).
	\end{equation*}
	Hence $|\mathbb P(\widehat{\eta}_0\in A)-\mathbb P(\widetilde{\eta}_1\in A)|=O(T^{\theta_2-1+\varepsilon})$, thus $d_{LP}(\widehat{\eta}_0,\widetilde{\eta}_1)=O(T^{\theta_2-1+\varepsilon})$. For Kantorovich---Rubinstein distance we can (in mass transportation interpretation) move all excessive mass in $\operatorname{Law}(\widehat{\eta}_0)$ to the origin and then back to the points where 
	$\operatorname{Law}(\widetilde{\eta}_1)$ has excessive mass. The total work for this transportation is
	\begin{equation*}
		\int |x|\, |\rho(x)-1|\,dF_{\widehat{\eta}_0}(x)=O(T^{\theta_2+\varepsilon-1}).
	\end{equation*}
	It remains to estimate the distances $d(\widetilde{\eta}_1,\widehat{\eta}_1)$. We use the identity
	\begin{equation*}
	\widehat{\eta}_1=\frac{\widetilde{\eta}_1-\mathbb{E}\widetilde{\eta}_1}{\sqrt{\Var\widetilde{\eta}_1}}
	\end{equation*}
	and the estimates
	\begin{equation*}
		|\mathbb E\widetilde{\eta}_1|=O(T^{\theta_2-1+\varepsilon}),\quad
		\Var \widetilde{\eta}_1=1+O(T^{\theta_2-1+\varepsilon})		
	\end{equation*}
	obtained above. This yields
	\begin{equation}\label{eq:d-tilde-hat}
	\begin{aligned}
	d_{KR}(\widetilde{\eta}_1,\widehat{\eta}_1)&=O(T^{\theta_2-1+\varepsilon}),\\
	d_{LP}(\widetilde{\eta}_1,\widehat{\eta}_1)&=O(T^{\frac23\bigl(\theta_2-1\bigr)+\varepsilon}).\qedhere
	\end{aligned}
	\end{equation}
\end{proof}

\begin{lemma}\label{lem:dist-12}
	For any $\beta>0$ we have
	\begin{align*}
	d_{KR}(\widehat{\eta}_1,\widehat{\eta}_2)&=O(n^{\theta_2^2-\theta_2+\beta}),\\
	d_{LP}(\widehat{\eta}_1,\widehat{\eta}_2)&=O(n^{\frac23\bigl(\theta_2^2-\theta_2\bigr)+\beta}).
	\end{align*}
\end{lemma}

\begin{proof}
	As in the previous lemma, consider
	\begin{equation*}
		\widetilde{\eta}_2=\frac{\eta_2-\mathbb{E}\eta_1}{\sqrt{\Var\eta_1}}
	\end{equation*}
	Then Corollary~\ref{cor:approx} yields
	\begin{equation*}
	|\widetilde{\eta}_2-\widehat{\eta}_1|\le \frac{|\eta_1-\eta_2|}{\sqrt{\Var\eta_1}}=O(n^{\theta_2^2-\theta_2+\delta}),
	\end{equation*}
	hence
	\begin{equation*}
	d_{KR,LP}(\widetilde{\eta}_2,\widehat{\eta}_1)=O(n^{\theta_2^2-\theta_2+\delta},
	\end{equation*}
	since for $d_{KR}$ we move every piece of mass by the distance at most $\alpha=\sup|\widetilde{\eta}_2-\widehat{\eta}_1|$, and for $d_{LP}$ we see that $\widehat{\eta}_1\in A$ implies $\widetilde{\eta}_2\in A^\alpha$ and vice versa.
	Further, 
	\begin{gather*}
	|\mathbb E\widetilde{\eta}_2|\le |\widehat{\eta}_1|+
	\mathbb E|\widetilde{\eta}_2-\widehat{\eta}_1|=O(n^{\theta_2^2-\theta_2+\delta}),\\
	|\mathbb E\widetilde{\eta}_2^2-\mathbb E\widehat{\eta}_1^2|\le
	2 \mathbb E|\widehat{\eta}_1|\cdot\alpha+\alpha^2=O(n^{\theta_2^2-\theta_2+\delta}),
	\end{gather*}
whence $\Var\widetilde{\eta}_2=1+O(n^{\theta_2^2-\theta_2+\delta})$.
As in the previous lemma, these estimates yields the same formulas as~\eqref{eq:d-tilde-hat} for $d(\widetilde{\eta}_2,\widehat{\eta}_2)$, and this concludes the proof.
\end{proof}

Theorem 3.1 now follows from Bufetov's limit theorem and Lemmas~\ref{lem:dist-01} and \ref{lem:dist-12}.

\subsubsection*{Acknowledgments} The work was partially supported by RFBR--CNRS grant No.~18-51-15010 and RFBR grant No.~18-31-20031. The author deeply thanks Alexander Bufetov for fruitful discussions on this matter.

\end{document}